\documentclass[review, sort&compress]{elsarticle}

\usepackage{lineno,hyperref}

\journal{arxiv}

\usepackage{amssymb}
\usepackage{amssymb,amsmath,amsthm,epsfig,multicol}
\newtheorem{theorem}{Theorem}[section]
\newtheorem{lemma}[theorem]{Lemma}
\newtheorem{corollary}[theorem]{Corollary}

\newtheorem{definition}[theorem]{Definition}
\newtheorem{example}[theorem]{Example}
\newtheorem{remark}{Remark}

\usepackage{xcolor}

\newcommand{\N}{\mathbb{N}}
\newcommand{\fallingfactorial}[1]{^{\underline{#1}}}
\usepackage{mathrsfs}
\usepackage{graphicx}
\usepackage{float}
\usepackage{mathtools}

\begin{document}
	
	\begin{frontmatter}

		\title{On Hilfer fractional difference operator}
			
			\author[]{Syed Sabyel Haider\corref{mycorrespondingauthor}}
		\author{Mujeeb ur Rehman, Thabet Abdeljawad }	
			\address[]{}
		\begin{abstract}
	In this article, a new definition of fractional Hilfer difference operator is introduced.
	Definition based properties are developed and utilized to construct fixed point operator for fractional order Hilfer difference equations with initial condition.
	We acquire some conditions for existence, uniqueness,   Ulam-Hyers and Ulam-Hyers-Rassias stability. Modified Gronwall's inequality is presented for discrete calculus with the delta difference operator.		 
		\end{abstract}
		
	\end{frontmatter}
	
\section{Introduction}
	In the topics of  discrete fractional calculus a variety of  results can be found in \cite{Abdel,DFC9,Atici,DFC15,DFC00,DFC102,DFC101,Goodrich,DFC18,Holm2,Holm,abdeljawad2018riemann} which has helped to construct theory of the subject.  A rigorous intrigue  in  fractional calculus of differences has been exhibited by Atici and Eloe \cite{Atici,DFC15}. They  explored characteristics of  falling function, a new  power law for  difference operators and the composition of  sums and differences of arbitrary order.  Holm presented advance composition formulas for  sums and differences  in his dissertation \cite{Holm}.

	Hilfer  fractional order  derivative was introduced in \cite{Hilfer}.  Hilfer's definition is illustrated as follows: the fractional derivative of order $0 < \mu < 1$ and type $0\leq \nu \leq 1$ is
	\begin{equation*}D_{a}^{\mu,\nu}f(x)=\Big(I_{a}^{\nu(1-\mu)}\frac{d}{dx} \big( I_{a}^{(1-\nu)(1-\mu)} f\big)\Big)(x).\end{equation*}
	The special cases are Riemann-Liouville fractional derivative   for  $\nu=0$ and Caputo 	fractional derivative  for  $\nu=1$. Furati et al. \cite{Furati,Furati2} primarily studied the existence theory  of Hilfer  fractional   derivative and also explained the type parameter $\nu$ as interpolation between the	Riemann-Liouville and the Caputo derivatives. It generates more types of stationary
	states and gives an extra degree of freedom on the initial condition.
	
	Hilfer fractional calculus has  been examined broadly by a lot of researchers.
	Some  recent studies involving Hilfer fractional derivatives can be found in \cite{Chitalkar,Rezazadeh,Sousa,Vivek2,Vivek3,Wang2}.
	Majority of the work in discrete fractional calculus developed as analogues of continuous fractional calculus.    Extensive work on  Hilfer fractional derivative and on its extensions has been done, namely: Hilfer-Hadamard \cite{Abbas,Kassim,Qassim,ahmad2019hyers}, K-fractional Hilfer \cite{Dorrego}, Hilfer-Prabhakar \cite{Garra}, Hilfer-Katugampola \cite{Oliveira}  and $\psi$-Hilfer \cite{Abdo} fractional operator. However, to the best of our knowledge  no work is available  for Hilfer fractional difference operator in the  delta fractional setting.	Also  formation of fractional difference operator is an important aspect in view of mathematical interest and numerical formulae as well as  the applications.   This motivate us to generalize the two existing  fractional difference operators namely, Riemann-Liouville and Caputo difference operator in Hilfer's sense.

	We start   by introducing a generalized difference operator analogous to Hilfer fractional derivative \cite{Hilfer}.  To keep the interpolative property of Hilfer fractional difference operators we carefully choose the starting points of fractional sums. Some  important composition properties are developed and utilized to construct fixed point operator for a new class of Hilfer fractional nonlinear difference equation with initial conditions involving Reimann-Liouville fractional sum. An  application of Brouwer's fixed point theorem gives us conditions for the  existence of solution for new class of Hilfer fractional nonlinear difference equation. For  uniqueness of solution  we apply Banach contraction principle.
	To solve   linear  fractional Hilfer difference  equation we use successive approximation  method and  then define the discrete Mittag-Leffler function in the delta difference setting.  Gronwall's inequality  for discrete calculus with the delta difference  operator is modified. An application of Gronwall's inequality has been given  for stability  of solution to fractional order Hilfer difference equations with different initial conditions.

	In the continuous setting extensive work on Ulam-Hyers-Rassias stability for non integer order differential equation has been done. The idea of Ulam-Hyers type stability is important to both pure and applied problems; especially in biology, economics and numerical analysis.  Rassias \cite{Rassias}  introduced continuity condition which produced an acceptable  stronger results. However in discrete fractional setting a limited work can be found  \cite{Chen,Jonnalagadda,Sabyel2020}. For Hilfer  delta difference equation,    conditions are acquired for  Ulam-Hyers  and Ulam-Hyers-Rassias stability  with illustrative example.  Interested reader may find some details on Ulam-Hyers-Rassias stability in \cite{Hyers,Jung,Rassias,Ulam}.
	
	In this article, we shall study initial value problem (IVP)  for following Hilfer fractional difference equation.   Let $\eta=\mu+\nu-\mu\nu$, then for  $0 < \mu < 1$ and $0\le \nu \le 1$, we have
	\begin{equation}\label{eq414}
	\begin{split}
	\begin{cases}
	\Delta_{a}^{\mu,\nu}u(x)+g(x+\mu-1, u(x+\mu-1))=0,~\text{for}~x \in \mathbb{N}_{a+1-\mu}, \\
	\Delta_{a}^{-(1-\eta)}u(a+1-\eta)=\zeta,  ~~~\zeta\in \mathbb{R}.
	\end{cases}
	\end{split}
	\end{equation}
	In  Section 2, we state few basic but important  results from discrete  calculus. In  third section, a new fractional Hilfer difference operator is introduced which interpolate  Riemann-Liouville and Caputo fractional differences, we also  develop some important properties of newly defined operator.    Conditions for existence, uniqueness and Ulam-Hayer   stability will be obtained in  Section 4. Last section comprise of modification and application  of discrete Gronwall’s
	inequality in delta  setting.
	
	\section{Preliminaries} Some  basics from discrete fractional calculus are given for later use in following sections.
	The functions we will consider are usually defined on  the set
	$ \mathbb{N}_{a} := \lbrace a, a+1, a+2,\cdots \rbrace,$ where $a \in \mathbb{R}$ is fixed.  Some times the set $ \mathbb{N}_{a} $ is called isolated time scale. Similarly the set $\mathbb{N}_{a}^{b} :=  \{ a, a+1, a+2,\cdots, b \}$ and $[a, b]_{\mathbb{N}_{a}}:=[a, b] \cap\mathbb{N}_{a}$ \cite{Feng} for $b=a+k, ~ k \in \mathbb{N}_{0}.$ The  jump operators   $\sigma(t)=t+1$, and $\rho(t)=t-1$  are forward and backward respectively for $t \in\mathbb{N}_{a}$. Furthermore, the set  $\mathcal{R} =\{p_{i}: 1+p_{i}(x) \neq 0 \}$  contains regressive  functions.
	\begin{definition}\emph{\cite{Bohner}}	 Assume $f: \mathbb{N}_{a}  \to \mathbb{R}$ and $b\leq c$ are in $\mathbb{N}_{a}$, then the  delta definite integral is defined  by
		\begin{equation*}		\int_{b}^{c}  f(x)\Delta x	=\sum_{x=b}^{c-1} f(x).	\end{equation*}
		Note that the value of integral $\int_{b}^{c}  f(x)\Delta x$,  depending on the set $\{b, b+1,\cdots,c-1\}$. Also we adopt the empty sum convention
		$		\sum_{x=b}^{b-k} f(x)=0, ~~\mbox{whenever} ~ k \in \mathbb{N}_{1}.	$
	\end{definition}
	
	\begin{definition}\emph{\cite{Goodrich}}\label{def201}	 Assume  $\mu>0$ and $f: \mathbb{N}_{a}  \to \mathbb{R}$. Then the   delta fractional sum of $f$ is defined by
		$\Delta_{a}^{-\mu}f(x):=\sum_{\tau=a}^{x-\mu} h_{\mu-1}(x,\sigma (\tau))
		f(\tau)$, for $x  \in \mathbb{N}_{a+\mu},$
		where $h_{\mu}(t,s)=\frac{{(t-s)\fallingfactorial{\mu}}}{{\Gamma(\mu +1 )}}$  is  $\mu^{th}$ fractional Taylor monomial based at $s$ and ${t\fallingfactorial{\mu}}$ is the generalized falling function.
	\end{definition}
	\begin{lemma}\emph{\cite{Goodrich}}\label{lem1} Assume $\nu\geq0$ and $\mu>0$. Then
		$\Delta_{a+\nu}^{-\mu}(x-a)\fallingfactorial{\nu}=\frac{\Gamma(\nu+1)}{\Gamma(\mu+\nu+1)}(x-a)\fallingfactorial{\mu+\nu}$, for $x  \in \mathbb{N}_{a+\mu+\nu}$.
	\end{lemma}
	\begin{definition}\label{def202}\emph{\cite{Atici,miller}}	 Assume $f:\mathbb{N}_{a}\to\mathbb{R}$, $\mu>0$ and $m-1<\mu\leq m$, for $m\in\mathbb{N}_{1}$ . Then the  Riemann-Liouville  fractional difference of $f$ at $a$ is defined by
		$$\Delta_{a}^{\mu} f(x)=\Delta^{m}\Delta_{a}^{-(m-\mu)}f(x)=\sum_{\tau=a}^{x+\mu} h_{-\mu-1}(x,\sigma(\tau))	
		f(\tau) , ~~\text{for~} x  \in \mathbb{N}_{a+m-\mu}.$$
	\end{definition}
	
	\begin{definition}\label{def203}\emph{\cite{Abdel,Abdel2}}	 Assume $f: \mathbb{N}_{a}  \to \mathbb{R}$, $\mu>0$ and $m-1<\mu\leq m$, for $m\in\mathbb{N}_{1}$. Then the  Caputo   fractional difference of $f$ at $a$ is defined by
		$$~^{c}\Delta_{a}^{\mu} f(x)=\Delta_{a}^{-(m-\mu)}\Delta^{m}f(x)=\sum_{\tau=a}^{x-(m-\mu)} h_{m-\mu-1}(x,\sigma(\tau))	
		\Delta^{m}f(\tau),$$~~\text{for~} $x \in \mathbb{N}_{a+m-\mu}.$
	\end{definition}
	
	\begin{definition}\emph{\cite{Goodrich}}	Assume $ p\in \mathcal{R}$ and $x, y\in \mathbb{N}_{a}$. Then the  delta exponential function is given by,		\[		e_{ p(x)}(x, y)=		\begin{cases}			\prod_{t=y}^{x-1} [1+p(t)], ~~~~&\text{if}~x \in  \N_{y},\\ 	\prod_{t=x}^{y-1} [1+p(t)]^{-1}, ~~~~&\text{if}~x \in  \N_{a}^{y-1}.		\end{cases}	\] By  empty product convention $\prod_{t=y}^{y-1} [h(t)]:=1$ for any function $h$.\end{definition}

	\begin{definition}\emph{\cite{Bohner} }\label{def126}	 Assume $f: \mathbb{N}_{a}  \to \mathbb{R}$.	Then  the delta Laplace transform of $f$ based at $a$ is defined  by	\[		\mathscr{L}_{a} \{ f\} (y)		=\int_{a}^{\infty} e_{\ominus y}(\sigma (x), a) f(x)\Delta x,	\]	for all complex numbers $ y\neq -1$  such that this improper integral converges.\end{definition}

	\begin{lemma}\emph{\cite{Goodrich}}\label{lem7} Assume $f: \mathbb{N}_{a}  \to \mathbb{R}$ is of exponential order $r>1$ and $\mu>0$. Then  	\[ \text{for~} |y+1|>r, \text{we have,}~	\mathscr{L}_{a+\mu} \{ \Delta_{a}^{-\mu}f \} (y) =\frac{(y+1)^{\mu}}{y^{\mu}}\tilde{ F}_{a}(y).	 \]\end{lemma}

	\begin{lemma}\emph{\cite{Goodrich}}\label{lem8} Assume that $f: \mathbb{N}_{a}  \to \mathbb{R}$ is of exponential order $r>0$ and   $m$ is positive integer.  Then for  $ |y+1|>r$	\[		\mathscr{L}_{a} \{ \Delta_{}^{m}f \} (y) =y^{m}\tilde{ F}_{a}(y)-\sum_{j=0}^{m-1} y^{j}\Delta_{}^{m-1-j} f(a).	 \]\end{lemma}

	\begin{lemma}\emph{\cite{Goodrich}}\label{lem202}(Fundamental theorem for the difference calculus)
		Assume $f: \mathbb{N}_{a}^{b}\to \mathbb{R}$ and $F$ is an antidifference of $f$ on $\mathbb{N}_{a}^{b+1}$. Then $\sum_{t=a}^{b}  f(t)=\sum_{t=a}^{b}  \Delta F(t)$ $=F(b+1)-F(a).$
	\end{lemma}
The definition of Ulam stability for fractional
difference equations is introduced in \cite{Chen}.
Consider the system \eqref{eq414} and the following inequalities:
\begin{equation}\label{eq501}
\Big{|}\Delta_{a}^{\mu,\nu}v(y)+g(y+\mu-1, v(y+\mu-1))     \Big{|} \le \epsilon,
~~~~y\in [a, T]_{\mathbb{N}_{a}},\end{equation}\vspace{-22pt}
\begin{equation}\label{eq502}
\Big{|}\Delta_{a}^{\mu,\nu}v(y)+g(y+\mu-1, v(y+\mu-1)) \Big{|} \le \epsilon \psi (\rho(y)+\nu),
~~~~ y\in [a, T]_{\mathbb{N}_{a}},\end{equation}
where $\psi: [a, T]_{\mathbb{N}_{a}}\to \mathbb{R}^{+}$.
\begin{definition}\emph{\cite{Chen}}\label{def205}	
	A solution 	$u\in Z$ of system \eqref{eq414} is Ulam-Hyers stable if there exists a real number $d_{f}>0$	 such that for each $\epsilon>0$ and for
	every solution $v \in Z$ of inequality \eqref{eq501}, if it satisfies
	\begin{equation}\label{eq503}
	\big{|}\big{|}v-u\big{|}\big{|} \le \epsilon d_{f}.\end{equation}
	A solution of system \eqref{eq414} is  generalized Ulam-Hyers stable if we
	substitute the function $\phi_{f}(\epsilon)$ for the constant $\epsilon d_{f}$ in inequality \eqref{eq503}, where $\phi_{f}(\epsilon)\in C(R^{+}, R^{+})$ and $\phi_{f}(0)=0$.
\end{definition}

\begin{definition}\emph{\cite{Chen}}\label{def206}	
	A solution
	$u\in Z$ of system \eqref{eq414}  is  Ulam-Hyers-Rassias stable with respect to  function  $\psi$  if there	exists a real number $d_{f,\psi}>0$
	such that for each $\epsilon>0$ and for
	every solution $v \in Z$ of inequality \eqref{eq502},  if it satisfies
	\begin{equation}\label{eq504}
	\big{|}\big{|}v-u\big{|}\big{|} \le \epsilon  \psi(y) d_{f,\psi},
	~~~~y\in [a, T]_{\mathbb{N}_{a}}.\end{equation}
	The solution of system \eqref{eq414} is  generalized Ulam-Hyers-Rassias stable if we
	substitute the function $\Phi(y)$ for the function $\epsilon \psi(y)$ in inequalities \eqref{eq502} and \eqref{eq504}.
\end{definition}
\section{Hilfer like fractional difference }
In this section, we give  generalize  definition of the Hilfer like fractional difference operator.  Motivated by the concept  of Hilfer fractional derivative \cite{Hilfer}, and to keep the interpolative property  we present the
following definition.
Assume $f: \mathbb{N}_{a}  \to \mathbb{R}$, then the fractional difference of order $m-1< \mu < m$, for $m\in\mathbb{N}_{1}$  is  given by
$	\Delta_{a}^{\mu,\nu}f(x)=\Delta_{a+(1-\nu)(m-\mu)}^{-\nu(m-\mu)}\Delta^{m} \Delta_{a}^{-(1-\nu)(m-\mu)} f(x),$
for $x  \in \mathbb{N}_{a+m-\mu}$, where $0\leq \nu \leq 1$ is the type of difference operator.  Observe that domain of
$\Delta_{a}^{-(1-\nu)(m-\mu)} f(x)$ is $a+(1-\nu)(m-\mu)$, whereas integer-order differences keeping the same domain \cite{Holm}. The starting point of the last sum  is compatible with the starting point for the domain of the  function $\Delta^{m} \Delta_{a}^{-(1-\nu)(m-\mu)} f(x)$, which is $a+(1-\nu)(m-\mu)$. This allow us the successive composition of operators in above expression and the final domain of $\Delta_{a}^{\mu,\nu}f(x)$ is $\mathbb{N}_{a+m-\mu}$.
To get some nice properties, we restrict   $0 < \mu < 1$ throughout the article.
\begin{definition} \label{def1} Assume $f: \mathbb{N}_{a}  \to \mathbb{R}$, then the fractional difference of order $0 < \mu < 1$ and type $0\leq \nu \leq 1$ is defined by
	\begin{equation*}\Delta_{a}^{\mu,\nu}f(x)=\Delta_{a+(1-\nu)(1-\mu)}^{-\nu(1-\mu)}\Delta  \Delta_{a}^{-(1-\nu)(1-\mu)} f(x),\end{equation*}
	for $x  \in \mathbb{N}_{a+1-\mu}$.			
\end{definition}
\noindent The special cases are Riemann-Liouville  fractional difference \cite{Atici,miller}   for  $\nu=0$ and Caputo fractional difference \cite{Abdel,Abdel2}	  for  $\nu=1$.

First we develop some composition properties to use in the next section, to construct a  fixed point operator for a new class of Hilfer fractional nonlinear difference equation with initial conditions involving Reimann-Liouville fractional sum. Also we present the  delta  Laplace transform for newly defined  Hilfer  fractional difference operator.
\begin{lemma}\label{lem301}
	Assume $f: \mathbb{N}_{a}  \to \mathbb{R}$,  $0 < \mu < 1$ and  $0\leq \nu \leq 1$, then for  $x\in N_{a+1}$
	\begin{itemize}	
		\item[$(i)$]	$\Delta_{a+1-\mu}^{-\mu}[\Delta_{a}^{\mu,\nu}f(x)]=\Delta_{a+(1-\nu)(1-\mu)}^{-(\mu+\nu-\mu\nu)}\Delta \Delta_{a}^{-(1-\nu)(1-\mu)} f(x),$
		\item[$(ii)$]	$\Delta_{a+1-\mu}^{-\mu}[\Delta_{a}^{\mu,\nu}f(x)]=\Delta_{a+(1-\nu)(1-\mu)}^{-(\mu+\nu-\mu\nu)} \Delta_{a}^{\mu+\nu-\mu\nu}f(x),$
		\item[$(iii)$]	$\Delta_{a+\mu}^{\mu,\nu}[\Delta_{a}^{-\mu}f(x)]=\Delta_{a+(1-\nu+\mu\nu)}^{-\nu(1-\mu)} \Delta_{a}^{\nu(1-\mu)}f(x),$
		\item[$(iv)$]	$\Delta_{a+\mu}^{\mu,\nu}[\Delta_{a}^{-\mu}f(x)]=
		f(x)-\Delta_{a}^{-(1-\nu(1-\mu))} f(a+1-\nu(1-\mu)) \times h_{\nu(1-\mu)-1}(x, a+$ $~~~~~~~~~~~~~~~~~~~~~~~~1-\nu(1-\mu)).$
	\end{itemize}
\end{lemma}
\begin{proof}
	$(i)$  On the  left hand side we  use Definition \ref{def1} and (Theorem 5 \cite{Holm}) to obtain
	\begin{equation*}\begin{split}
	\Delta_{a+1-\mu}^{-\mu}[\Delta_{a}^{\mu,\nu}f(x)]=&\Delta_{a+1-\mu}^{-\mu}[\Delta_{a+(1-\nu)(1-\mu)}^{-\nu(1-\mu)}\Delta  \Delta_{a}^{-(1-\nu)(1-\mu)} f(x)]
	\\=&\Delta_{a+(1-\nu)(1-\mu)}^{-(\mu+\nu-\mu\nu)}\Delta  \Delta_{a}^{-(1-\nu)(1-\mu)} f(x).
	\end{split}\end{equation*}
	$(ii)$  On the left hand side, use $(i)$  and  first part of (Lemma 6 \cite{Holm}),
	\begin{equation*}\begin{split}
	\Delta_{a+1-\mu}^{-\mu}[\Delta_{a}^{\mu,\nu}f(x)]=&\Delta_{a+(1-\nu)(1-\mu)}^{-(\mu+\nu-\mu\nu)}\Delta  \Delta_{a}^{-(1-\nu)(1-\mu)} f(x)\\=&\Delta_{a+(1-\nu)(1-\mu)}^{-(\mu+\nu-\mu\nu)} \Delta_{a}^{\mu+\nu-\mu\nu}f(x).
	\end{split}\end{equation*}	
	
	\noindent$(iii)$  Using Definition \ref{def1} and (Theorem 5 \cite{Holm}) we get
	\begin{equation*}\begin{split}
	\Delta_{a+\mu}^{\mu,\nu}[\Delta_{a}^{-\mu}f(x)]=&\Delta_{a+\mu+(1-\nu)(1-\mu)}^{-\nu(1-\mu)}\Delta    \Delta_{a+\mu}^{-(1-\nu)(1-\mu)}[\Delta_{a}^{-\mu}f(x)]
	\\=&\Delta_{a+(1-\nu+\mu\nu)}^{-\nu(1-\mu)}\Delta  \Delta_{a}^{-(1-\nu+\mu\nu)}f(x)
	\\=&\Delta_{a+(1-\nu+\mu\nu)}^{-\nu(1-\mu)} \Delta_{a}^{\nu(1-\mu)}f(x).
	\end{split}\end{equation*}	
	In preceding step we also used first part of (Lemma 6 \cite{Holm}).\\ \noindent
	$(iv)$  Consider the left hand side, use $(iii)$  and  second part of (Theorem 8 \cite{Holm}),
	\begin{equation*}\begin{split}
	\Delta_{a+\mu}^{\mu,\nu}[\Delta_{a}^{-\mu}f(x)]=
	&\Delta_{a+(1-\nu+\mu\nu)}^{-\nu(1-\mu)} \Delta_{a}^{\nu(1-\mu)}f(x)\\=&\Delta_{a+1-\nu(1-\mu)}^{-\nu(1-\mu)} \Delta_{a}^{\nu(1-\mu)}f(x)\\=&	
	f(x)-\Delta_{a}^{-(1-\nu(1-\mu))} f(a+1-\nu(1-\mu)) \\&\times h_{\nu(1-\mu)-1}(x, a+1-\nu(1-\mu)).\qedhere
	\end{split}\end{equation*}	
\end{proof}
For nonempty set $N_{a}^{T}$, the set of all real valued bounded functions  $B(N_{a}^{T})$ is a norm space with $||f||= \sup_{x\in \mathbb{N}_{a}^{T}}\{f(x)\}.$ We consider  a weighted space of bounded functions $B_{\lambda}(N_{a}^{T}):=\{f:N_{a}^{T} \to \mathbb{R}; | (x-a-\mu)\fallingfactorial{\lambda} f(x)|< M \}$, with $0\le \lambda< \mu$ and $M>0.$ The  weighted space of bounded functions is considered for finding left inverse property, however analysis in the following sections is not influenced by this space.
\begin{lemma}\label{lem302} 	
	Let $f \in B_{\lambda}(N_{a}^{T})$ be given and $0<\lambda \le 1$. Then
	$\Delta_{a}^{-\mu}f(a+\mu)=0,$ for $0\le \lambda <\mu.$
\end{lemma}
\begin{proof}
	Since $f \in B_{\lambda}(N_{a}^{T})$, thus for some positive integer $M,$ we have
	$ | (x-a-\mu)\fallingfactorial{\lambda} f(x)|< M,$ for each $x \in N_{a}^{T}.$ Therefore
	\begin{equation*}
	\begin{split} |\Delta_{a}^{-\mu}f(x)|
	<& M [\Delta_{a}^{-\mu} (y-a-\mu)\fallingfactorial{-\lambda} ](x)
	\\\le& M\Gamma(1-\lambda) \frac{(x-a-\mu)\fallingfactorial{\mu-\lambda}}{\Gamma(\mu-\lambda+1)}.
	\end{split}
	\end{equation*}
	In  the preceding step we used the fact  $\Delta_{a}^{-\mu} (x-a)\fallingfactorial{-\lambda}=(x-a)\fallingfactorial{\mu-\lambda}\frac{\Gamma(1-\lambda)}{\Gamma(\mu-\lambda+1)}.$ The desired result is achieved by applying limit process $x \to a+\mu$.
\end{proof}
\noindent Next we state the left inverse property.
\begin{lemma}\label{lem303}
	Assume $0 < \mu < 1$,  $0\leq \nu \leq 1$ and $\eta=\mu+\nu-\mu\nu$, then for  $f \in B_{1-\eta}(N_{a}^{T}),$
	$$\Delta_{a+\mu}^{\mu,\nu}[\Delta_{a}^{-\mu}f(x)]=f(x).$$
\end{lemma}
\begin{proof} Since $0\le 1-\eta < 1-\nu(1-\mu)$. Thus Lemma \ref{lem302} gives
	$\Delta_{a}^{-(1-\nu+\mu\nu)} f(a+1-\nu+\mu\nu)=0.$ Hence the result follows from the part $(iv)$ of Lemma \ref{lem301}.
\end{proof}

\begin{theorem}	\label{thm9} Assume $f: \mathbb{N}_{a}  \to \mathbb{R}$ is of exponential order $r>1$ with $	\mathscr{L}_{a} \{ f(x)\} (y)$ $=\tilde{ F}_{a}(y)$ and $0 < \mu < 1$, $0\leq \nu \leq 1$. Then for $ |y+1|>r$ we have  the delta Laplace transform given as
	\begin{equation*}\begin{split}	\mathscr{L}_{a+1-\mu} \{\Delta_{a}^{\mu, \nu}f \} (y) =& y^{\mu}(y+1)^{1-\mu}\tilde{ F}_{a}(y)\\&-\frac{(y+1)^{\nu(1-\mu)}}{y^{\nu(1-\mu)}} \Delta_{a}^{-(1-\nu)(1-\mu)} f(a+(1-\nu)(1-\mu)).	
	\end{split}\end{equation*}
\end{theorem}
Proof:   Considering the left hand  side  and using the Lemmas,	\ref{lem7} and 	\ref{lem8},	
\begin{equation*}\begin{split}
\mathscr{L}_{a+1-\mu} \{\Delta_{a}^{\mu, \nu}f \} (y) =&\mathscr{L}_{a+1-\mu} [\Delta_{a+(1-\nu)(1-\mu)}^{-\nu(1-\mu)}\Delta \Delta_{a}^{-(1-\nu)(1-\mu)} f(x)	] (y)\\=&\frac{(y+1)^{\nu(1-\mu)}}{y^{\nu(1-\mu)}} \mathscr{L}_{a+(1-\nu)(1-\mu)} [\Delta \Delta_{a}^{-(1-\nu)(1-\mu)} f(x)	] (y)\\=&\frac{(y+1)^{\nu(1-\mu)}}{y^{\nu(1-\mu)}}\Big[ y\mathscr{L}_{a+(1-\nu)(1-\mu)} [ \Delta_{a}^{-(1-\nu)(1-\mu)} f(x)	] (y)\\&- \Delta_{a}^{-(1-\nu)(1-\mu)} f(a+(1-\nu)(1-\mu))	\Big]
\\=&\frac{(y+1)^{\nu(1-\mu)}}{y^{\nu(1-\mu)}}\Big[ y\frac{(y+1)^{(1-\nu)(1-\mu)}}{y^{(1-\nu)(1-\mu)}}\mathscr{L}_{a} [ f(x)	] (y)\\&- \Delta_{a}^{-(1-\nu)(1-\mu)} f(a+(1-\nu)(1-\mu))	\Big]
\\=& y^{\mu}(y+1)^{1-\mu}\tilde{ F}_{a}(y)\\&-\frac{(y+1)^{\nu(1-\mu)}}{y^{\nu(1-\mu)}} \Delta_{a}^{-(1-\nu)(1-\mu)} f(a+(1-\nu)(1-\mu)).	
\end{split}\end{equation*}
\begin{remark}
	Notice that, if in Theorem 	\ref{thm9}, we set $\nu=0$ then we recover Theorem 2.70 in \cite{Goodrich}. Further, if we set $\nu=1$ we obtain the delta Laplace transform for the Caputo fractional difference.
\end{remark}
\section{Fixed point operators for initial value problem}
To establish existence theory  for Hilfer fractional difference equation with initial conditions, we transforms the problem to an equivalent summation equation which in turn defined an appropriate fixed point operator.
\begin{lemma}\label{lem402} 	
	Let $g:[a,T]_{\mathbb{N}_{a}}\times\mathbb{R}\to \mathbb{R}$ be given and $0<\mu < 1$, $0\le \nu \le 1$. Then $u$ solves
	system \eqref{eq414} if and only if	
	\begin{equation*}
	\begin{split}
	u(x)=\zeta h_{\eta-1}(x,a+1-\eta)- \Delta_{a+1-\mu}^{-\mu}
	g(x+\mu-1, u(x+\mu-1)),
	\end{split}
	\end{equation*}~for ~all~ $x\in \mathbb{N}_{a+1}.$
\end{lemma}	
\noindent The proof of above lemma is an implication of  Lemma \ref{lem301} $(i)$ and $(ii)$  and  second part of Theorem 8 in \cite{Holm}.
In next result the  Brouwer's fixed Point theorem \cite{Chen} is utilized for establishing  existence conditions.
The set $Z$ of all real sequences  $u=\{u(x)\}_{x=a}^{T},$ with $||u||=\adjustlimits\sup_{x\in \mathbb{N}_{a}^{T}}|u(x)|$ is a Banach space.\\
Using Definition \ref{def201} and Lemma \ref{lem402} we define an  operator
$\mathcal{A}:Z \to Z$ by
\begin{equation}\label{eq401}\begin{split}\mathcal{A}u(x)
=&\zeta h_{\eta-1}(x,a+1-\eta)-
\sum_{\tau=a+1-\mu}^{x-\mu} h_{\mu-1}(x,\sigma (\tau))
g(\tau+\mu-1, u(\tau+\mu-1)).   \end{split}\end{equation}
The  fixed points of  $\mathcal{A}$  coincides with the solutions of the  problem \eqref{eq414}.
\begin{theorem}\label{thm402}	Let 	 $f:[a,T]_{\mathbb{N}_{a}}\to \mathbb{R}$ be  a bounded function in such a way that $|g(x,u)|\leq f(x)|u|$ for all $u \in Z.$ Then IVP \eqref{eq414} has at least one solution on $Z$, provided
	\begin{equation}\label{eq402}L^{*}\leq \frac{\Gamma(\mu +1 )}{(T-a-1+\mu)\fallingfactorial{\mu}},\end{equation}
	where $L^{*}=\adjustlimits\sup_{x\in \mathbb{N}_{a+1-\mu}^{T}}f(x+\mu-1)$.
\end{theorem}

\begin{proof}
	For $M>0$, define the set
	$$W=\{u:  ||u- \zeta h_{\eta-1}(x,a+1-\eta) ||\leq M,~ \text{for}~ x\in \mathbb{N}_{a+1-\mu}^{T} \}. $$
	To prove this theorem  we just have to show that $\mathcal{A}$ maps $W$ into itself.
	For $u \in W$, we have $\Big{|}\mathcal{A}u(x)-\zeta  h_{\eta-1}(x,a+1-\eta)\Big{|}$
	\begin{equation*}\begin{split}
	\leq&  f(x+\mu-1)\sum_{\tau=a+1-\mu}^{x-\mu} h_{\mu-1}(x,\sigma (\tau))  |u(\tau+\mu-1)-0| \\
	\leq& L^{*} \adjustlimits\sup_{x\in \mathbb{N}_{a+1-\mu}^{T}}|u(x+\mu-1)-0|\sum_{\tau=a+1-\mu}^{x-\mu} h_{\mu-1}(x,\sigma (\tau))   \\
	=& L^{*} ||u-0|| \Big[\frac{(x-a-1+\mu)\fallingfactorial{\mu}}{\Gamma(\mu +1 )}-0\Big]    \\
	\le& L^{*} M \frac{(T-a-1+\mu)\fallingfactorial{\mu}}{\Gamma(\mu +1 )}\le M.   \end{split}\end{equation*}
	We have   $||\mathcal{A}u||\leq M$
	which
	implies that $\mathcal{A}$ is self map. Therefore by Brouwer's fixed point
	theorem  $\mathcal{A}$ has at least one fixed point.
\end{proof}

\begin{theorem}\label{thm403}	For $K>0$ and $u,v \in Z$ assume that    $|g(x,u)- g(x,v)|\leq K|u-v|$,  for all  $x\in [a,T]_{\mathbb{N}_{a}}.$  Then IVP \eqref{eq414} has  unique solution on $Z$, provided
	\begin{equation}\label{eq403}K <\frac{\Gamma(\mu +1 )}{(T-a-1+\mu)\fallingfactorial{\mu}}.\end{equation}
\end{theorem}
\begin{proof}
	Let $u,v \in Z$ and	 $x\in [a,T]_{\mathbb{N}_{a}}$, we have by assumption
	\begin{equation*}\begin{split}
	\Big{|}\mathcal{A}u(x)-\mathcal{A}v(x)\Big{|}\leq&  \Big{|}
	\sum_{\tau=a+1-\mu}^{x-\mu} h_{\mu-1}(x,\sigma (\tau))\Big{|}
	\\ &\times | g(\tau+\mu-1, u(\tau+\mu-1))-g(\tau+\mu-1, v(\tau+\mu-1))|
	\\ \leq& \frac{ |0-(x-a-1+\mu)\fallingfactorial{\mu}|
	}{\Gamma(\mu+1)}  K|  u(\tau+\mu-1)- v(\tau+\mu-1)|.   \end{split}\end{equation*}
	In the preceding step, we used  $\sum_{\tau} h_{\nu-1}(x,\sigma (\tau))=-h_{\nu}(x,\tau)$ and Lemma \ref{lem202}.
	Now taking supremum on both sides we have
	$$\sup_{x\in \mathbb{N}_{a}^{T}}\Big{|}\mathcal{A}u(x)-\mathcal{A}v(x)\Big{|}\leq\frac{ K(T-a-1+\mu)\fallingfactorial{\mu}
	}{\Gamma(\mu+1)}  ||u-v ||.$$
	Using  inequality  \eqref{eq403}, we get   $||\mathcal{A}u-\mathcal{A}v||\leq ||u-v ||$
	which implies $\mathcal{A}$ is contraction. Therefore by Banach fixed point theorem   $\mathcal{A}$ has unique fixed point.
\end{proof}	

\noindent To illustrate the usefulness of Theorem , we present the following example.							
\begin{example} Consider the  fractional Hilfer difference equation with  initial conditions involving Reimann-Liouville fractional sum
	\begin{equation*}
	\begin{split}
	\begin{cases}
	-\Delta_{0.3}^{0.7, 0.5}u(x)=(x-0.3)u(x-0.3),
	~~~~~~~ x\in [0.3, 9.3]_{\mathbb{N}_{0.3}}\\
	\Delta_{0.3}^{-(0.15)}u(0.45)=\zeta.
	\end{cases}\end{split}
	\end{equation*}
	Here $a=0.3$,  $T=9.3$, $\mu=0.7$ and $\nu=0.5$. Therefore $\eta=0.85$. Thus  for  $K< 0.1974$, the solution to the given problem  with inequalities \begin{equation*}
	\Big{|}\Delta_{0.3}^{0.7, 0.5}v(x)+(x-0.3) v(x-0.3)\Big{|} \le \epsilon
	~~~~x\in [0.3, 9.3]_{\mathbb{N}_{0.3}},\end{equation*}
	\begin{equation*}
	\Big{|}\Delta_{0.3}^{0.7, 0.5}v(x)+(x-0.3) v(x-0.3)\Big{|} \le \epsilon \psi (x-0.3)
	~~~~ x\in [0.3, 9.3]_{\mathbb{N}_{0.3}},\end{equation*} is  Ulam-Hyers stable  and Ulam-Hyers-Rassias stable with respect to  function  $\psi: [0.3, 9.3]_{\mathbb{N}_{0.3}}\to \mathbb{R}^{+}$.
\end{example}
To solve the
linear   Hilfer fractional difference  IVP  we use the successive
approximation
method.
\begin{example} \label{Ex4.7} Let $\eta=\mu+\nu-\mu\nu$, with  $0 < \mu < 1$ and $0\le \nu \le 1$. Consider  the IVP for linear Hilfer fractional difference equation,
	\begin{equation}\label{eq425}
	\begin{split}
	\begin{cases}
	\Delta_{a}^{\mu,\nu}u(x)-\lambda u(x+\mu-1)=0, \\
	\Delta_{a}^{-(1-\eta)}u(a+1-\eta)=\zeta,  ~~~\zeta\in \mathbb{R}.
	\end{cases}
	\end{split}
	\end{equation}
	The solution of \eqref{eq425} is given by	
	\begin{equation*}
	\begin{split} u(x)
	=&\zeta h_{\eta-1}(x,a+1-\eta)+\lambda\Delta_{a+1-\mu}^{-\mu}
	u(x+\mu-1).
	\end{split}
	\end{equation*} 	
	Definition \ref{def201} and successive approximation yield the following
	\begin{equation}\label{eq426}
	\begin{split} u_{k}(x)
	=&u_{0}(x)+\lambda \sum_{\tau=a+1-\mu}^{x-\mu} h_{\mu-1}(x,\sigma (\tau))
	u_{k-1}(\tau+\mu-1),
	\end{split}
	\end{equation}
	for   $k=1, 2, 3, \cdots,$ where $u_{0}(x)=\zeta h_{\eta-1}(x,a+1-\eta).$
	\\ Initially for $k=1$ and by Lemma \ref{lem1}
	\begin{equation*}
	\begin{split} u_{1}(x)
	=&\zeta h_{\eta-1}(x,a+1-\eta)+\lambda \zeta h_{\eta-1+\mu}(x+\mu-1,a+1-\eta).
	\end{split}
	\end{equation*}
	Similarly for $k=2$
	\begin{equation*}
	\begin{split} u_{2}(x)
	=&\zeta \Big{[} h_{\eta-1}(x,a+1-\eta)+\lambda h_{\eta-1+\mu}(x+\mu-1, a+1-\eta) +{\lambda}^2  h_{\eta-1+2\mu}(x\\&+2(\mu-1), a+1-\eta)\Big{]}\\=&
	\zeta \Big{[} {\lambda}^0 \frac{(x+\eta-a-1)\fallingfactorial{0.\mu+\eta-1}}{\Gamma(\eta )}
	+{\lambda}^1 \frac{(x+\eta-a-1+(\mu-1))\fallingfactorial{1.\mu+\eta-1}}{\Gamma(\mu+\eta )}\\&+{\lambda}^2 \frac{(x+\eta-a-1+2(\mu-1))\fallingfactorial{2.\mu+\eta-1}}{\Gamma(2\mu+\eta )}\Big{]}.
	\end{split}
	\end{equation*}
	Proceeding inductively and let $k\to \infty$
	\begin{equation*}
	\begin{split} u(x)
	=&
	\zeta \Big{[}  \sum_{k=0}^{\infty} {\lambda}^k \frac{(x+\eta-a-1+k(\mu-1))\fallingfactorial{k\mu+\eta-1}}{\Gamma(k\mu+\eta )}\Big{]}.\end{split}
	\end{equation*}
	Now 	we use property $x\fallingfactorial{\mu+\nu}=(x-\nu)\fallingfactorial{\mu}~x\fallingfactorial{\nu}$ in the following  step,
	\begin{equation*}
	\begin{split}
	u(x)=&
	\zeta \Big{[}  \sum_{k=0}^{\infty} {\lambda}^k \frac{(x+\eta-a-1+(k-1)(\mu-1))\fallingfactorial{k\mu}(x+\eta-a-1+k(\mu-1))\fallingfactorial{\eta-1}}{\Gamma(k\mu+\eta )}\Big{]}.	
	\end{split}
	\end{equation*}	
	Now from the discrete form \eqref{eq426} we have numerical formula
	\begin{equation}\label{eq427}
	\begin{split} u(a+n)
	=&u(a)+ \frac{\lambda}{\Gamma(\mu)} \sum_{j=1}^{n} \frac{\Gamma(n-j+\mu)}{\Gamma(n-j+1)}
	u(a+j-1),
	\end{split}
	\end{equation}
	with $u(a)=\zeta \frac{\Gamma(n+\eta)}{\Gamma(\eta)\Gamma(n+1)}.$
	From  \eqref{eq427}, we can have
	\begin{equation*}
	\begin{split} y(n)
	=&\zeta \frac{\Gamma(n+\eta)}{\Gamma(\eta)\Gamma(n+1)}+ \frac{\lambda}{\Gamma(\mu)} \sum_{j=1}^{n} \frac{\Gamma(n-j+\mu)}{\Gamma(n-j+1)}
	y(j-1).
	\end{split}
	\end{equation*}
	For different values of $\nu$ numerical solutions for  $\mu=0.8$ and $\mu=0.5$  are shown in  Fig. 1 and  Fig. 2 respectively. Fig. 1 and  Fig. 2 show the 	interpolative behavior of Hilfer difference operator  between the	Riemann-Liouville \cite{DFC102} and the Caputo difference operator \cite{wu}.
\end{example}
\begin{remark}\label{rem1}
	If we set $\nu=1$ in Example \ref{Ex4.7} above (hence $\eta=1$), and take $a=\mu-1$, then we recover Example 17 in \cite{Abdel}. In fact, the solution of the initial Caputo difference equation
	\begin{equation}\label{ICDE}
	~^{C}\Delta_a^\mu x(t)=\lambda x(t+\mu-1),~~x(a)=x_0,~~\mu \in (0,1],
	\end{equation}
	will be given by
	\begin{equation}\label{srep}
	x(t)=x_0 E_{\underline{\mu}}(\lambda,t-a)=x_0 \sum_{k=0}^\infty \frac{\lambda^k (t-a+k(\mu-1))^{\underline{k \mu}}}{\Gamma(\mu k+1)}.
	\end{equation}
	Observe that the case $a=\mu-1$ will result in (66) in \cite{Abdel}. That is, the formula (66) in \cite{Abdel} represents 
	$E_{\underline{\mu}}(\lambda,t-(\alpha-1))$. Also, one can see that the substitution $\mu=1$ will give the delta discrete Taylor expansion of the delta discrete  exponential function.
\end{remark}
The observations in Remark \ref{rem1}, suggest the following modified alternative definitions which are different from that in \cite{Abdel}.
\begin{definition}\label{def207} For $\lambda \in \mathbb{R},~|\lambda|<1$ and $\mu,\eta,\gamma, z \in \mathbb{C}$ with $Re(\mu)>0,$
	the	discrete  Mittag-Leffler functions are defined by

	\begin{eqnarray}
	\nonumber 
	E^{\gamma}_{\fallingfactorial{\mu,\eta}}(\lambda,  z) &=& \sum_{k=0}^\infty \lambda^k\frac{(z+k(\mu-1))^{\underline{\mu k+\eta-1}} (\gamma)_k}{\Gamma(\mu k+\eta)k!},~(\gamma)_k=\gamma (\gamma+1)\cdots(\gamma+k-1),
	\end{eqnarray}
	
	\begin{equation}\label{nn}
	E_{\fallingfactorial{\mu,\eta}}(\lambda,  z)=E^{1}_{\fallingfactorial{\mu,\eta}}(\lambda,  z)=\sum_{k=0}^\infty \lambda^k\frac{(z+k(\mu-1))^{\underline{\mu k+\eta-1}} }{\Gamma(\mu k+\eta)},
	\end{equation}
	\begin{equation}\label{nn}
	E_{\fallingfactorial{\mu}}(\lambda,  z)=E_{\fallingfactorial{\mu,1}}(\lambda,  z)=\sum_{k=0}^\infty \lambda^k\frac{(z+k(\mu-1))^{\underline{\mu k}} }{\Gamma(\mu k+1)}.
	\end{equation}
\end{definition}
\noindent By the help of the fact that $x\fallingfactorial{\mu+\nu}=(x-\nu)\fallingfactorial{\mu}~x\fallingfactorial{\nu}$ , we note that 
\begin{eqnarray}\label{bbb}
E^\gamma_{\fallingfactorial{\mu,\mu}}(\lambda,  z)&=&\sum_{k=0}^\infty \lambda^k\frac{(z+k(\mu-1))^{\underline{\mu k+\mu-1}(\gamma)_k} }{\Gamma(\mu k+\mu)k!}\\ \nonumber
&=& \sum_{k=0}^{\infty} {\lambda}^k \frac{(z+(k-1)(\mu-1))\fallingfactorial{k\mu}(z+k(\mu-1))\fallingfactorial{\mu-1} (\gamma)_k}{\Gamma(k\mu+\mu )k!}.
\end{eqnarray}

%
\begin{definition}\label{def2088} For $\lambda \in \mathbb{R},~|\lambda|<1$ and $\mu,\eta,\gamma, z \in \mathbb{C}$ with $Re(\mu)>0,$
	the	discrete  Mittag-Leffler functions are defined by	
	
	\begin{eqnarray}
	\nonumber
	\textbf{E}^{\gamma}_{\fallingfactorial{\mu,\eta}}(\lambda,  z) &=& \sum_{k=0}^\infty \lambda^k\frac{(z+k(\mu-1)+\eta-1)^{\underline{\mu k+\eta-1}} (\gamma)_k}{\Gamma(\mu k+\eta)k!},
	\end{eqnarray}
	
	\begin{equation}\label{nn}
	\textbf{E}_{\fallingfactorial{\mu,\eta}}(\lambda,  z)=\textbf{E}^{1}_{\fallingfactorial{\mu,\eta}}(\lambda,  z)=\sum_{k=0}^\infty \lambda^k\frac{(z+k(\mu-1)+\eta-1)^{\underline{\mu k+\eta-1}} }{\Gamma(\mu k+\eta)},
	\end{equation}
	\begin{equation}\label{nn}
	\textbf{E}_{\fallingfactorial{\mu}}(\lambda,  z)=\textbf{E}_{\fallingfactorial{\mu,1}}(\lambda,  z)=E_{\fallingfactorial{\mu}}(\lambda,  z)=\sum_{k=0}^\infty \lambda^k\frac{(z+k(\mu-1))^{\underline{\mu k}} }{\Gamma(\mu k+1)}.
	\end{equation}
\end{definition}


Next we solve the non-homogeneous
Hilfer fractional difference  IVP, which shows that Definition \label{def208} is useful.

\begin{example} Let $\eta=\mu+\nu-\mu\nu$, with  $0 < \mu < 1$ and $0\le \nu \le 1$. Consider Hilfer non-homogeneous fractional difference equation,
	\begin{equation}\label{eq415}
	\begin{split}
	\begin{cases}
	\Delta_{a}^{\mu,\nu}u(x)-\lambda u(x+\mu-1)=f(x), \\
	\Delta_{a}^{-(1-\eta)}u(a+1-\eta)=\zeta,  ~~~\zeta\in \mathbb{R}.
	\end{cases}
	\end{split}
	\end{equation}
	The solution of \eqref{eq415} is given by	
	\begin{equation*}
	\begin{split} u(x)
	=&\zeta h_{\eta-1}(x,a+1-\eta)+\lambda\Delta_{a+1-\mu}^{-\mu}
	u(x+\mu-1)+\Delta_{a+1-\mu}^{-\mu}
	f(x).
	\end{split}
	\end{equation*} 	
	Then, Definition \ref{def201} and successive approximation yields the following
	\begin{equation*}
	\begin{split} u_{k}(x)
	=&u_{0}(x)+\lambda \sum_{\tau=a+1-\mu}^{x-\mu} h_{\mu-1}(x,\sigma (\tau))
	u_{k-1}(\tau+\mu-1)
	+\Delta_{a+1-\mu}^{-\mu}
	f(x),
	\end{split}
	\end{equation*}
	for   $k=1, 2, 3, \cdots,$ where $u_{0}(x)=\zeta h_{\eta-1}(x,a+1-\eta).$
	\\ Initially for $k=1$ and by Lemma \ref{lem1}
	\begin{equation*}
	\begin{split} u_{1}(x)
	=&\zeta h_{\eta-1}(x,a+1-\eta)+\lambda \zeta h_{\eta-1+\mu}(x+\mu-1,a+1-\eta)
	+\Delta_{a+1-\mu}^{-\mu}f(x).
	\end{split}
	\end{equation*}
	Similarly for $k=2$
	\begin{equation*}
	\begin{split} u_{2}(x)
	=&\zeta \Big{[} h_{\eta-1}(x,a+1-\eta)+\lambda h_{\eta-1+\mu}(x+\mu-1, a+1-\eta) +{\lambda}^2  h_{\eta-1+2\mu}(x\\&+2(\mu-1), a+1-\eta)\Big{]}+\lambda\Delta_{a+1-\mu}^{-2\mu}f(x+\mu-1)
	+\Delta_{a+1-\mu}^{-\mu}f(x)\\=&
	\zeta \Big{[} {\lambda}^0 \frac{(x+\eta-a-1)\fallingfactorial{0.\mu+\eta-1}}{\Gamma(\eta )}
	+{\lambda}^1 \frac{(x+\eta-a-1+(\mu-1))\fallingfactorial{1.\mu+\eta-1}}{\Gamma(\mu+\eta )}\\&+{\lambda}^2 \frac{(x+\eta-a-1+2(\mu-1))\fallingfactorial{2.\mu+\eta-1}}{\Gamma(2\mu+\eta )}\Big{]}+\lambda\Delta_{a+1-\mu}^{-2\mu}f(x+\mu-1)
	\\&+\Delta_{a+1-\mu}^{-\mu}f(x).
	\end{split}
	\end{equation*}
	Proceeding inductively and let $k\to \infty$
	\begin{equation*}
	\begin{split} u(x)
	=&
	\zeta \Big{[}  \sum_{k=0}^{\infty} {\lambda}^k \frac{(x+\eta-a-1+k(\mu-1))\fallingfactorial{k\mu+\eta-1}}{\Gamma(k\mu+\eta )}\Big{]}\\&+\sum_{k=1}^{\infty}{\lambda}^{k-1} \Delta_{a+1-\mu}^{-k\mu}f(x+(k-1)(\mu-1))\\=&\zeta \Big{[}  \sum_{k=0}^{\infty} {\lambda}^k \frac{(x+\eta-a-1+k(\mu-1))\fallingfactorial{k\mu+\eta-1}}{\Gamma(k\mu+\eta )}\Big{]}\\&+\sum_{k=1}^{\infty}{\lambda}^{k-1}
	\sum_{\tau=a+1-\mu}^{x-k\mu} h_{k\mu-1}(x,\sigma (\tau+(k-1)(\mu-1)))f(\tau)\\=&\zeta \Big{[}  \sum_{k=0}^{\infty} {\lambda}^k \frac{(x+\eta-a-1+k(\mu-1))\fallingfactorial{k\mu+\eta-1}}{\Gamma(k\mu+\eta )}\Big{]}\\&+\sum_{k=0}^{\infty}{\lambda}^{k}
	\sum_{\tau=a+1-\mu}^{x-k\mu-\mu} \frac{(x-\sigma (\tau)+k(\mu-1))\fallingfactorial{k\mu+\mu-1}}{\Gamma(k\mu+\mu )}f(\tau)
	\end{split}
	\end{equation*}
	\begin{equation*}
	\begin{split} u(x)
	=&
	\zeta \Big{[}  \sum_{k=0}^{\infty} {\lambda}^k \frac{(x+\eta-a-1+k(\mu-1))\fallingfactorial{k\mu+\eta-1}}{\Gamma(k\mu+\eta )}\Big{]}\\&+\sum_{\tau=a+1-\mu}^{x-\mu}\sum_{k=0}^{\infty}{\lambda}^{k}
	\frac{(x-\sigma (\tau)+k(\mu-1))\fallingfactorial{k\mu+\mu-1}}{\Gamma(k\mu+\mu )}f(\tau).\end{split}
	\end{equation*}
	In preceding step, we have interchanged summation of second expression. Now
	
	we use property $x\fallingfactorial{\mu+\nu}=(x-\nu)\fallingfactorial{\mu}~x\fallingfactorial{\nu}$ in the following  step,
	\begin{equation*}
	\begin{split}
	u(x)=&
	\zeta \Big{[}  \sum_{k=0}^{\infty} {\lambda}^k \frac{(x+\eta-a-1+(k-1)(\mu-1))\fallingfactorial{k\mu}(x+\eta-a-1+k(\mu-1))\fallingfactorial{\eta-1}}{\Gamma(k\mu+\eta )}\Big{]}\\&+\sum_{\tau=a+1-\mu}^{x-\mu}\sum_{k=0}^{\infty}{\lambda}^{k}
	\frac{(x-\sigma (\tau)+(k-1)(\mu-1))\fallingfactorial{k\mu}(x-\sigma (\tau)+k(\mu-1))\fallingfactorial{\mu-1}}{\Gamma(k\mu+\mu )}f(\tau).\end{split}
	\end{equation*}
	Using Definition \ref{def207}, we have	
	\begin{equation*}	u(x)=
	\zeta E_{\fallingfactorial{\mu,\eta}}(\lambda,  x+\eta-a-1) +\sum_{\tau=a+1-\mu}^{x-\mu}
	\Big{[}  E_{\fallingfactorial{\mu,\mu}}(\lambda, x-\sigma (\tau)) \Big{]}f(\tau).\end{equation*}
	Alternatively, by using Definition \ref{def2088}
	\begin{equation*}	u(x)=
	\zeta \textbf{E}_{\fallingfactorial{\mu,\eta}}(\lambda,  x-a) +\sum_{\tau=a+1-\mu}^{x-\mu}
	\Big{[} \textbf{ E}_{\fallingfactorial{\mu,\mu}}(\lambda, x-\sigma (\tau)+\mu-1) \Big{]}f(\tau).\end{equation*}
	
\end{example}
\noindent Note that above is the generalization  of Caputo fractional difference IVP \cite {Abdel}, one can prevail it for $\nu=1.$

\section{ Modified Gronwall's inequality and its application in delta difference setting}

First we develop a Gronwall's
inequality for  the delta difference operator.  Then
a simple utilization  of  Gronwall's inequality leads to   stability for Hilfer  difference equation.
For this purpose choose $u$ and $w$ such that
\begin{equation}\label{eq601}
u(x)
\le u(a) h_{\eta-1}(x,a+1-\eta)+\Delta_{a+1-\mu}^{-\mu}
\phi(x+\mu)u(x+\mu),
\end{equation}
\begin{equation}\label{eq602}
w(x)
\ge w(a) h_{\eta-1}(x,a+1-\eta)+\Delta_{a+1-\mu}^{-\mu}
\phi(x+\mu)w(x+\mu).
\end{equation}

\begin{lemma}\label{lem500}
	Assume $u$ and  $w$ are respectively satisfying \eqref{eq601} and \eqref{eq602}. If $w(a)\ge u(a)$, then $w(x)\ge u(x)$ for $x \in  \mathbb{N}_{a}.$
\end{lemma}
\begin{proof} We give the proof by  induction principle.
	Assume $w(\tau)-u(\tau)\ge 0$ is valid for $\tau= a, a + 1,  \cdots, x-1.$ Then we have
	\begin{equation*}
	\begin{split} w(x)- u(x)
	\ge& h_{\eta-1}(x,a+1-\eta)(w(a)-u(a))+\Delta_{a+1-\mu}^{-\mu}
	\phi(x+\mu)w(x+\mu)\\&-\Delta_{a+1-\mu}^{-\mu}
	\phi(x+\mu)u(x+\mu)\\=&h_{\eta-1}(x,a+1-\eta)(w(a)-u(a))
	\\&+\sum_{\tau=a+1-\mu}^{x-\mu}
	\frac{(x-\sigma (\tau))\fallingfactorial{\mu-1}}{\Gamma(\mu )}
	\phi(\tau+\mu)(w(\tau+\mu)-u(\tau+\mu)),
	\end{split}
	\end{equation*}
	where the last summation is valid for   $x \in  \mathbb{N}_{a+\mu}.$ Now we shift the domain of summation to   $  \mathbb{N}_{a}.$
	\begin{equation*}
	\begin{split} w(x)- u(x)
	\ge&h_{\eta-1}(x,a+1-\eta)(w(a)-u(a))
	\\&+\sum_{\tau=a+1}^{x}
	\frac{(x+\mu-\sigma (\tau))\fallingfactorial{\mu-1}}{\Gamma(\mu )}
	\phi(\tau)(w(\tau)-u(\tau)).
	\end{split}
	\end{equation*}
	By  assumption,	for $\tau= a, a + 1, \cdots, x-1,$ we have
	\begin{equation*}
	\begin{split} w(x)- u(x)
	\ge &
	\phi(x)(w(x)-u(x)).
	\end{split}
	\end{equation*}
	This implies that $(1-\phi(x))(w(x)-u(x))\ge0$	and for  $|\phi(x)|<1,$  which is  desired result.
	
\end{proof}
Following the approach for nabla fractional difference in \cite{Atici2}, let $E_{v}\phi=\Delta_{a+1-\mu}^{-\mu}v(x)\phi(x).$ For constant $\phi$ one can use $E_{v}\phi$  to express Mittag-Leffler function.
\begin{theorem}\label{thm501}
	Assume $\eta=\mu+\nu-\mu\nu$, with  $0 < \mu < 1$ and $0\le \nu \le 1$.
	The solution of summation equation	
	\begin{equation*}
	\begin{split} u(x)
	=& u(a) h_{\eta-1}(x,a+1-\eta)+\Delta_{a+1-\mu}^{-\mu}
	v(x+\mu-1)u(x+\mu-1),
	\end{split}
	\end{equation*}
	is given by	
	\begin{equation*}
	\begin{split} u(x)
	=&\frac{u(a)}{\Gamma(\eta)}
	\sum_{\ell=0}^{\infty}	
	E_{v}^{\ell}(x+\eta-a-1+\ell(\mu-1))\fallingfactorial{\eta-1}.
	\end{split}
	\end{equation*}
\end{theorem}
\begin{proof}
	By method of successive approximation the following  is obtained
	\begin{equation*}
	\begin{split} u_{k}(x)
	=&u_{0}(x)+\Delta_{a+1-\mu}^{-\mu}v(x+\mu-1)	u_{k-1}(x+\mu-1),~~  k=1, 2, 3, \cdots,
	\end{split}
	\end{equation*}
	where $u_{0}(x)=u(a) h_{\eta-1}(x,a+1-\eta).$
	\\For $k=1$
	\begin{equation*}
	\begin{split} u_{1}(x)
	=&u(a) h_{\eta-1}(x,a+1-\eta)+\Delta_{a+1-\mu}^{-\mu}
	v(x+\mu-1) u_{0}(x+\mu-1)
	\\=&\frac{u(a)}{\Gamma(\eta)}E_{v}^{0}(x+\eta-a-1)\fallingfactorial{\eta-1}+\frac{u(a)}{\Gamma(\eta)}E_{v}^{1}(x+\eta-a-1+\mu-1)\fallingfactorial{\eta-1}.
	\end{split}
	\end{equation*}
	Proceeding inductively we get
	\begin{equation*}
	\begin{split} u_{k}(x)
	=&\frac{u(a)}{\Gamma(\eta)}
	\sum_{\ell=0}^{k}	
	E_{v}^{\ell}(x+\eta-a-1+\ell(\mu-1))\fallingfactorial{\eta-1}, ~~  k=1, 2, 3, \cdots,
	\end{split}
	\end{equation*}
	and let $k\to \infty,$
	\begin{equation*}
	\begin{split} u(x)
	=&\frac{u(a)}{\Gamma(\eta)}
	\sum_{\ell=0}^{\infty}	
	E_{v}^{\ell}(x+\eta-a-1+\ell(\mu-1))\fallingfactorial{\eta-1}.
	\end{split}
	\end{equation*}
\end{proof}
\noindent Next we derive  a Gronwall's inequality in delta discrete setting.
\begin{theorem}\label{thm502}
	Let $\eta=\mu+\nu-\mu\nu$, with  $0 < \mu < 1$ and $0\le \nu \le 1$. Assume $|v(x)|<1$ for $x \in  \mathbb{N}_{a}.$
	If $u$ and  $v$ are nonnegative real valued functions with
	\begin{equation*}
	\begin{split} u(x)
	\le&u(a) h_{\eta-1}(x,a+1-\eta)+\Delta_{a+1-\mu}^{-\mu}
	v(x+\mu-1)u(x+\mu-1).
	\end{split}
	\end{equation*}
	Then
	\begin{equation*}
	\begin{split} u(x)
	\le&\frac{u(a)}{\Gamma(\eta)}
	\sum_{\ell=0}^{\infty}	
	E_{v}^{\ell}(x+\eta-a-1+\ell(\mu-1))\fallingfactorial{\eta-1}.
	\end{split}
	\end{equation*}
\end{theorem}
\begin{proof}
	Consider $w(x)=\frac{u(a)}{\Gamma(\eta)}
	\sum_{\ell=0}^{\infty}	
	E_{v}^{\ell}(x+\eta-a-1+\ell(\mu-1))\fallingfactorial{\eta-1}.$ The  proof of  theorem follows from Lemma \ref{lem500} and Theorem \ref{thm501}.
\end{proof}
For $\eta=1,$ a special case is obtained as follow.
\begin{corollary}\label{cor503}
	Let  $0 < \mu < 1$ and $0\le \nu \le 1$. Assume $0<v(x)<1$ for $x \in  \mathbb{N}_{a}.$
	If $u$ is nonnegative real valued function with
	\begin{equation*}
	\begin{split} u(x)
	\le&u(a) +\Delta_{a+1-\mu}^{-\mu}
	v(x+\mu-1)u(x+\mu-1).
	\end{split}
	\end{equation*}
	Then
	\begin{equation*}
	\begin{split} u(x)
	\le& u(a)e_{v}(x,a),
	\end{split}
	\end{equation*}
	where $e_{v}(x,a)$ is the delta exponential function.
\end{corollary}
\begin{proof}
	It follows from Theorem \ref{thm502} that
	\begin{equation*}u(x)
	\le u(a) \sum_{\ell=0}^{\infty}	
	E_{v}^{\ell}(1).\end{equation*}
	We claim that $ \sum_{\ell=0}^{\infty}	
	E_{v}^{\ell}(1)=e_{v}(x,a).$
	To justify our claim, we utilize the uniqueness  of solution   of  following IVP,
	$\Delta u(x) = v(x)u(x), ~u(a) = 1$. A unique solution $u(x)=e_{v}(x,a)$ of IVP  is given in \cite{Goodrich} for  regressive function $v(x)$.
	Thus, we  have to show that $ \sum_{\ell=0}^{\infty}	
	E_{v}^{\ell}(1)$
	satisfies the IVP $\Delta u(x) = v(x)u(x), ~u(a) = 1.$ Indeed,
	\begin{equation*}
	\begin{split}\Delta \sum_{\ell=0}^{\infty}	
	E_{v}^{\ell}(1)=& \sum_{\ell=0}^{\infty}\Delta 	
	E_{v}^{\ell}(1)\\=&  \sum_{\ell=1}^{\infty}\Delta 	
	E_{v}(E_{v}^{\ell-1}(1))\\=& \sum_{\ell=1}^{\infty}\Delta 	
	\Delta_{a}^{-1} (v(x)E_{v}^{\ell-1}(1))=v(x)\sum_{\ell=0}^{\infty}E_{v}^{\ell}(1).
	\end{split}
	\end{equation*}
	Also by Definition \ref{def201} and empty sum convention we have $\sum_{\ell=0}^{\infty}	
	E_{v}^{\ell}(1)(a)=1+\sum_{\ell=1}^{\infty}	
	E_{v}^{\ell}(1)(a)=1.$ Then the result follows.
\end{proof}
\noindent Let $\eta=\mu+\nu-\mu\nu$, then for  $0 < \mu < 1$ and $0\le \nu \le 1$, we have  $0  < \eta \le  1.$  Following result illustrates the application of Gronwall's inequality, for the system
\begin{equation}\label{eq416}
\begin{split}
\begin{cases}
\Delta_{a}^{\mu,\nu}v(x)+g(x+\mu-1, v(x+\mu-1))=0,\text{for}~x \in \mathbb{N}_{a+1-\mu}, \\
\Delta_{a}^{-(1-\eta)}v(a+1-\eta)=\xi,  ~~~\xi\in \mathbb{R}.
\end{cases}
\end{split}
\end{equation}
\begin{theorem}
	Assume Lipschitz condition 	   $|g(x,u)- g(x,v)|\leq K|u-v|$ holds for function $g$.  Then the solution to Hilfer fractional difference system is  stable.
\end{theorem}
\begin{proof}
	Let $u\in Z$  be a solution of system \eqref{eq414} and  $v\in Z$  be a solution of system \eqref{eq416}. Then the corresponding summation equations are
	\begin{equation*}
	\begin{split} u(x)
	=&\zeta h_{\eta-1}(x,a+1-\eta)- \Delta_{a+1-\mu}^{-\mu}
	g(x+\mu-1, u(x+\mu-1)),
	\end{split}
	\end{equation*}
	\begin{equation*}
	\begin{split} v(x)
	=&\xi h_{\eta-1}(x,a+1-\eta)- \Delta_{a+1-\mu}^{-\mu}
	g(x+\mu-1, v(x+\mu-1)).
	\end{split}
	\end{equation*}
	For the  absolute value of the difference  we have  $|u(x)-v(x)|$
	\begin{equation*}
	\begin{split}
	\le&|\zeta-\xi| |h_{\eta-1}(x,a+1-\eta)|\\&+ |\Delta_{a+1-\mu}^{-\mu}
	(g(x+\mu-1, u(x+\mu-1))-
	g(x+\mu-1, v(x+\mu-1)))|
	\\\le&|\zeta-\xi| h_{\eta-1}(x,a+1-\eta)+ \Delta_{a+1-\mu}^{-\mu}
	K|u(x+\mu-1)- v(x+\mu-1)|.
	\end{split}
	\end{equation*}
	Then it follows from the  Theorem \ref{thm502}  that
	\begin{equation*}
	\begin{split}  |u(x)-v(x)|
	\le&
	\frac{|\zeta-\xi|}{\Gamma(\eta)}
	\sum_{\ell=0}^{\infty}	
	E_{K}^{\ell}(x+\eta-a-1+\ell(\mu-1))\fallingfactorial{\eta-1}.
	\end{split}
	\end{equation*}
	By using Lemma \ref{lem1} we get, $E_{K}^{\ell}(x+\eta-a-1+\ell(\mu-1))\fallingfactorial{\eta-1}=  \frac{K^\ell\Gamma(\eta)}{\Gamma(\eta+\mu\ell)} (x+\eta-a-1+\ell(\mu-1))\fallingfactorial{\eta+\mu\ell-1}.$  To shape in the form of discrete  Mittag-Leffler function we use property $x\fallingfactorial{\mu+\nu}=(x-\nu)\fallingfactorial{\mu}~x\fallingfactorial{\nu},$
	\begin{equation*}
	\begin{split}  |u(x)-v(x)|
	\le&
	|\zeta-\xi|
	\sum_{\ell=0}^{\infty} \frac{K^\ell}{\Gamma(\eta+\mu\ell)} (x+\eta-a-1+(k-1)(\mu-1))\fallingfactorial{k\mu}\\& \times (x+\eta-a-1+k(\mu-1))\fallingfactorial{\eta-1}\\=&
	|\zeta-\xi|
	E_{\fallingfactorial{\mu,\eta}}(K,  x+\eta-a-1),
	\end{split}
	\end{equation*}
	where  $E_{\fallingfactorial{\mu,\eta}}(\lambda,  x)$ is discrete  Mittag-Leffler functions defined in \cite{Abdel}.
	Replace system \eqref{eq416} with
	\begin{equation}\label{eq418}
	\begin{split}
	\begin{cases}
	\Delta_{a}^{\mu,\nu}v(x)+g(x+\mu-1, v(x+\mu-1))=0,~~ \\
	\Delta_{a}^{-(1-\eta)}v(a+1-\eta)={\zeta}_{n},
	\end{cases}
	\end{split}
	\end{equation}
	for $x \in \mathbb{N}_{a+1-\mu}$  and $ {\zeta}_{n} \to \zeta.$  The solutions are denoted by $v_{n}$. Now we have
	\begin{equation*}
	\begin{split}  |u(x)-v_{n}(x)|
	\le&
	|\zeta-{\zeta}_{n}|
	E_{\fallingfactorial{\mu,\eta}}(K,  x+\eta-a-1).
	\end{split}
	\end{equation*}
	This leads to $|u(x)-v_{n}(x)|\to 0,$ when $ {\zeta}_{n} \to \zeta$ for $n\to \infty.$  This complete the proof.
\end{proof}
\section*{Conclusion}
We finish by concluding:\begin{itemize}
	\item A new  definition of Hilfer like  fractional difference  on discrete time scale  has been presented.
	\item The  delta  Laplace transform has been presented for newly defined  Hilfer  fractional difference operator.
	\item We have investigated a new class of Hilfer like fractional nonlinear difference equation with initial conditions involving Reimann-Liouville fractional sum.
	\item In particular, condition for the  existence, uniqueness  and two types of stabilities, called Ulam-Hyers stability and Ulam-Hyers-Rassias stability has been obtained.
	\item The  linear   Hilfer fractional difference equation with initial conditions  is solved and  alternative versions  of discrete  Mittag-Leffler functions are presented in comparison to \cite{Abdel}.	
	\item A Gronwall's inequality has been presented and applied for discrete calculus with the delta operator.
\end{itemize}


\section*{References}

\newpage
\begin{figure}[htb]
	\centering{
		\resizebox*{13cm}{!}{\includegraphics{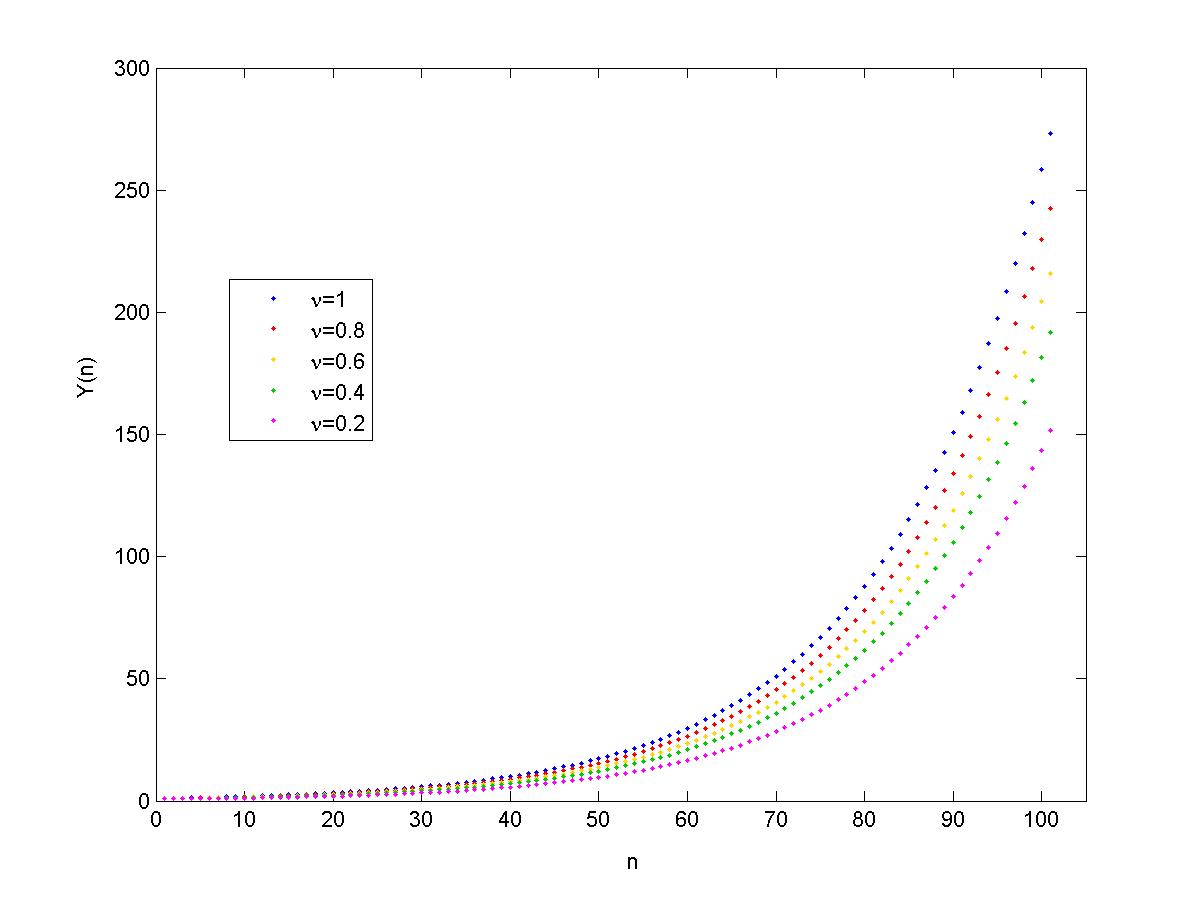}}}
	\caption{Solutions for $\lambda=0.1$, $\mu=0.8$ and different values of $\nu$. } \label{figure1}
	\centering{
		\resizebox*{13cm}{!}{\includegraphics{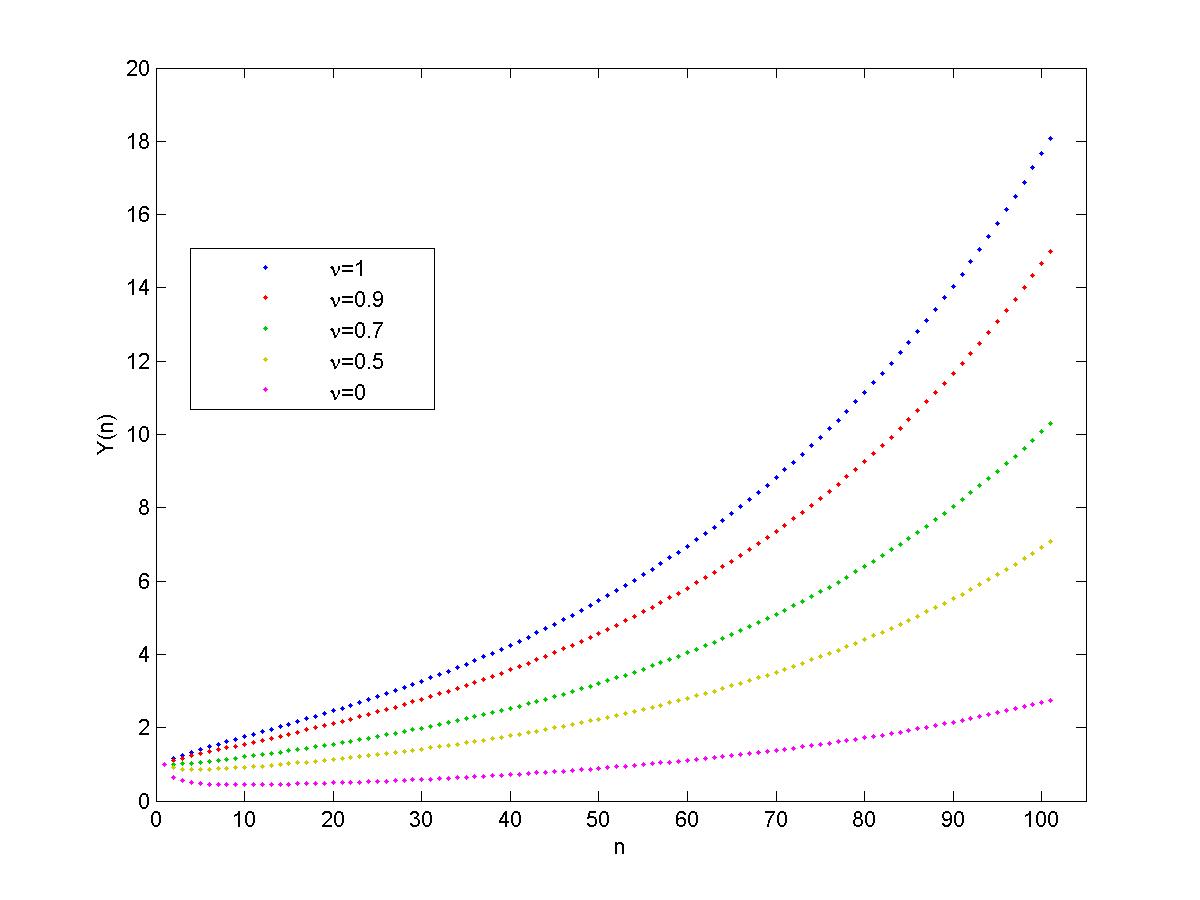}}}
	\caption{Solutions for $\lambda=0.1$, $\mu=0.5$ and different values of $\nu$. } \label{figure2}
\end{figure}

\end{document}